\documentclass[11pt,a4paper,reqno]{amsart}
\usepackage{amsmath,amssymb,amsfonts,epsfig,mathrsfs,cite, hyperref}
\usepackage[T1]{fontenc}
\usepackage{color}
\usepackage{array}
\usepackage{amsthm}
\usepackage{amstext}
\usepackage{amssymb}
\usepackage{graphicx}
\usepackage{setspace}

\usepackage[title]{appendix}

\makeatletter
\@namedef{subjclassname@2020}{%
  \textup{2020} Mathematics Subject Classification}
\makeatother

\usepackage[margin=2.5cm]{geometry}
\usepackage{color}
\usepackage{enumitem}
\usepackage{amscd,psfrag}
\usepackage{yhmath}

\usepackage{comment}

\allowdisplaybreaks[4]

\usepackage{slashed}

\makeatletter
\pdfpageheight\paperheight
\pdfpagewidth\paperwidth

\usepackage{epstopdf}
\usepackage{chngcntr}
\counterwithin{figure}{section}
\usepackage{mathrsfs}

\usepackage{indentfirst}	

\usepackage[normalem]{ulem}
\theoremstyle{plain}
\numberwithin{equation}{section}

\newtheorem{definition}{Definition}[section]
\newtheorem{theorem}[definition]{Theorem}
\newtheorem{defthm}[definition]{Definition/Theorem}
\newtheorem*{theorem*}{Theorem}

\newtheorem*{remark*}{Remark}
\newtheorem*{sideremark*}{Side Remark}

\newtheorem*{claim*}{Claim}
\newtheorem*{q*}{Question}
\newtheorem{lemma}[definition]{Lemma}

\newtheorem*{corollary*}{Corollary}

\newtheorem{proposition}[definition]{Proposition}

\newcommand{\R}{\mathbb{R}}

\newcommand{\p}{\partial}
\newcommand{\loc}{{\rm loc}}
\newcommand{\weak}{\rightharpoonup}
\newcommand{\e}{\varepsilon}
\newcommand{\C}{\mathbb{C}}
\newcommand{\dd}{{\rm d}}

\newcommand{\dvg}{{\,\dd V_g}}
\newcommand{\G}{\Gamma}

\newcommand{\1}{{\mathbf{1}}}

\def\XXint#1#2#3{{\setbox0=\hbox{$#1{#2#3}{\int}$ }
\vcenter{\hbox{$#2#3$ }}\kern-.6\wd0}}

\newcommand{\A}{{\mathscr{A}}}

\newcommand{\pdo}{\Psi {\rm DO}}

\newcommand{\ppl}{{\sigma_{\rm ppl}(\A)}}

\newcommand{\euc}{{\delta_{\rm Euc}}}

        \newcommand{\opa}{{\rm Op}(a)}

\title{A Compensated compactness theorem for pseudodifferential operators on vector bundles}

\author{Siran Li}
\address{Siran Li: School of Mathematical Sciences $\&$ CMA-Shanghai, Shanghai Jiao Tong University, No.~6 Science Buildings,
800 Dongchuan Road, Minhang District, Shanghai, China (200240)}
\email{\texttt{siran.li@sjtu.edu.cn}}

\author{Xiangxiang Su}
\address{Xiangxiang Su: School of Mathematical Sciences, Shanghai Jiao Tong University, No.~6 Science Buildings,
800 Dongchuan Road, Minhang District, Shanghai, China (200240)}
\email{\texttt{sjtusxx@sjtu.edu.cn}}

\author{Yuantu Zhu}
\address{Yuantu Zhu: School of Mathematical Sciences, Shanghai Jiao Tong University, No.~6 Science Buildings,
800 Dongchuan Road, Minhang District, Shanghai, China (200240)}
\email{\texttt{radonzhu@sjtu.edu.cn}}


\date{\today}

\begin{document}

\begin{abstract}
We establish a compensated compactness theorem in the  microlocal and geometric analytic framework. For a weakly $L^2_\loc$-convergent sequence of sections of a vector bundle over a semi-Riemannian manifold whose image under a pseudo-differential operator $\mathscr{A}$ of order $s>0$ is precompact in $H^{-s}_\loc$, we show that a quadratic form $Q$ acting on this sequence converges in the distributional sense, provided that $Q$ vanishes on the operator cone of $\mathscr{A}$. This extends the classical Murat--Tartar theory of compensated compactness from constant-coefficient first-order differential constraints on Euclidean spaces to variable-coefficient pseudo-differential constraints of arbitrary order on semi-Riemannian manifolds. 
\end{abstract}
\maketitle

\section{Introduction}

The theory of compensated compactness, pioneered by Murat \cite{Murat} and Tartar \cite{Tartar}, has played a fundamental role in the analysis of nonlinear partial differential equations (PDE) and calculus of variations; \emph{cf.} \cite{Dafermos, Evans, H, R, Tartar2, GQ, EM, DR}. It serves as a central tool for establishing \emph{weak continuity} results, \emph{i.e.}, determining conditions under which nonlinear combinations of weakly convergent sequences of (approximate) solutions to PDE or variational problems also converge, which in general fails in the absence of suitable notions of convexity.

The prototypical result in the theory of compensated compactness is the \emph{div-curl lemma} \emph{\`{a} la} Murat \cite{Murat} $\&$ Tartar \cite{Tartar}. In its simplest form, it reads as follows:
\begin{theorem*}
    Let $\{v_\e\}$, $\{w_\e\}$ be two sequences of vector fields on $\R^3$, such that $v_\e \weak v$ and $w_\e \weak w$  weakly in $L^2_{\rm loc}(\R^3;\R^3)$. Suppose that $\{{\rm div}(v_\e)\}$ is precompact in $H^{-1}_{\rm loc}(\R^3)$ and $\{{\rm curl}(w_\e)\}$ is precompact in $H^{-1}_{\rm loc}(\R^3;\R^3)$. Then $v_\e\cdot w_\e$ converges to $v \cdot w$ in the sense of distributions. 
\end{theorem*}

The above result ascertains that the dot product of weakly $L^2$-convergent sequences converges in the distributional sense, provided that compactness is assumed for suitable first-order differential operators acting on these sequences. Generalisations to differential forms, to sequences in $L^p$--$L^{p'}$ spaces, and to more general differential constraints other than those for div and curl, can be found in \cite{RRT, KY2, Evans, CL'}, among other references.

From a different perspective, Tartar established the following (\cite[pp.270--271, Corollary~2]{Tartar}, known as Tartar's quadratic theorem), from which the div-curl lemma can be deduced. 
\begin{theorem}\label{thm: Tartar}
    Let $\{u_\e\}$ be a sequence of maps in $L^2_{\rm loc}(\R^n; \R^p)$ such that $u_\e \weak u$ in $L^2_\loc$, let $\{a_{ijk}:1\leq i \leq q,\, 1 \leq j \leq p,\,1\leq k \leq n\}$ be constants, and let $Q$ be a quadratic form in $\R^p$. 
    Suppose that $\left\{\sum_{j=1}^p\sum_{k=1}^n a_{ijk}\frac{\p(u_\e)^j}{\p x^k}\right\}$ is precompact in $H^{-1}_\loc(\R^n)$ for each $1 \leq i \leq q$. Also suppose that $Q(\lambda)=0$ for every $\lambda \in \Lambda$, where \begin{align*}
     \Lambda :=  \left\{\lambda \in \R^p:\, \exists \xi \in \R^n\setminus \{0\} \text{ such that } \sum_{j=1}^p\sum_{k=1}^n a_{ijk}\lambda^j \xi^k=0 \text{ for all } 1\leq i \leq q \right\}.
    \end{align*} 
    Then $Q(u_\e)$ converges to $Q(u)$ in the sense of distributions.
\end{theorem}

Over the past two decades, the theory of compensated compactness  has been continuously investigated and enriched. Coifman--Lions--Meyer--Semmes developed a harmonic analytic framework based on Hardy-BMO duality and commutator estimates in their seminal paper \cite{CLMS}, while functional analytic (\cite{W, P, KY1, KY2, CL'}) and geometric analytic approaches (via Hodge decomposition \cite{RRT, Evans}) have also been explored. Recently, compensated compactness theorems have been established in fractional Sobolev spaces, with applications to nonlocal geometric PDE and variational problems \cite{MS}.


In this contribution, we establish one more compensated compactness result, Theorem~\ref{thm:psdo-cc}, by combining ideas from both microlocal and geometric analysis. It is close in spirit to Tartar's quadratic Theorem~\ref{thm: Tartar}, but features highly nontrivial generalisations in the following aspects:
\begin{enumerate}
    \item
    It treats a general pseudodifferential operator  ($\pdo$) of arbitrary order $s>0$ instead of the first-order constant-coefficient differential operator $\A$ with $(\A u)_i = a_{ijk}\p_ku^j$.
    \item 
    The weakly convergent terms $u_\e$ are allowed to be sections of vector bundle over a semi-Riemannian manifold, rather than mappings between Euclidean spaces.
\end{enumerate}
$\pdo$s are central tools for studying the propagation of singularities in nonlinear PDEs \cite{Chen, AG, Hormander}, so it seems natural to expect them to play a role in the compensated compactness theory.

\begin{theorem}\label{thm:psdo-cc}
Let $(M^n,g)$ be a $C^\infty$-semi-Riemannian manifold, let $E$ and $F$ be real $C^\infty$-vector bundles over $M$, let $\A\in \Psi^s(M;E,F)$ be a pseudodifferential operator of order $s>0$, and let $Q$ be a quadratic form on $E$. Assume that a sequence $\{u_\varepsilon\}_{\varepsilon>0}\subset L^2_{\mathrm{loc}}(M;E)$ satisfies:
\begin{enumerate}[label=(C\arabic*)]
\item 
$u_\varepsilon\rightharpoonup \bar u$ weakly in $L^2_{\mathrm{loc}}(M;E)$;
\item $\{\A u_\varepsilon\}$ is precompact in $H^{-s}_{\mathrm{loc}}(M;F)$;
\item $Q(v)=0$ for all $v\in \Lambda_\A$.
\end{enumerate}
Then $Q(u_\e)$ converges to $Q(\bar{u})$ in the sense of distributions.
\end{theorem}

The \emph{operator cone} $\Lambda_\A$ of the $\pdo$ $\A$ is defined as follows. Throughout, we denote by $\ppl$ the \emph{principal symbol} of $\A$, which is understood as a Hom-bundle ${\rm Hom}(E;F^\C)$-valued degree-$s$-homogeneous function on the cotangent bundle $T^*M$. Our proposed definition of the operator cone $\Lambda_\A$ is key to the formulation of Theorem~\ref{thm:psdo-cc} as a natural, concise, and geometric extension of Tartar's  quadratic Theorem~\ref{thm: Tartar}.

\begin{definition}\label{def: cone}
Let $\A\in \Psi^s(M;E,F)$ be a $\pdo$ of order $s \in \R$ between vector bundles $E$, $F$ over a semi-Riemannian manifold $(M,g)$. For each $x\in M$, define
$$
\Lambda_\A(x) := \left\{\lambda\in E_x:\, \exists\,\xi\in T_x^*M\setminus\{0\} \text{ such that } \sigma_{\mathrm{ppl}}(\A)(x,\xi)(\lambda)=0 \right\}.
$$
Then we set $$\Lambda_\A := \{v\in\Gamma(E) : v(x)\in\Lambda_\A(x) \text{ for all } x\in M\}.$$
\end{definition}

In \cite{CL}, the first-named author and Chen proposed the formulation of $\Lambda_\A$ as in Definition~\ref{def: cone}, and proved Theorem~\ref{thm:psdo-cc} for (constant-coefficient) differential operators $\A \in {\bf Diff}^k(M;E,F)$; $k \in \mathbb{N}$. Our paper should also be compared with Mi\v{s}ur--Mitrovi\'{c} \cite[Theorem~14; Corollary~15]{MM} (\emph{cf.} dissertation \cite{Misur} also), which formulates and proves a quadratic theorem in the setting $L^p$ -- $L^q$ and in the framework of $H$-distributions/$H$-measures introduced by Tartar \cite{T3} (also see Panov \cite{panov}). We speculate that our Theorem~\ref{thm:psdo-cc} might be equivalent to \cite[Corollary~15]{MM} when $p=q=2$, in the special case that $\A$ is a Fourier multiplier operator. In the future, we hope to further explore the links between microlocal analytic compensated compactness theorems and those in the framework of $H$-distributions.

We refer the reader to \S\ref{sec: prelim} for background knowledge of semi-Riemannian manifolds, $\pdo$s on vector bundles, and symbolic calculus for $\pdo$s. Theorem~\ref{thm:psdo-cc} shall be proved in \S\S\ref{sec: Eucl} and \ref{sec: mfd} on Euclidean spaces and vector bundles over semi-Riemannian manifolds, respectively.

\section{Preliminaries}\label{sec: prelim}


\subsection{$\pdo$s on Euclidean domains}
Throughout, $\mathcal{S}(\R^n)$ is the space of Schwartz (\emph{i.e.}, rapidly decreasing) functions.  See  \cite[I.2.1, I.3.1, I.7.1]{AG} for the following.

\begin{definition}

\begin{itemize}
   \item Let $n=1,2,3,\ldots$ and $m \in \mathbb{R}$. The symbol class $S^m(\mathbb{R}^n)$ consists of all functions $a=a(x,\xi) \in  C^\infty(\mathbb{R}^n \times \mathbb{R}^n)$ such that for any multi-indices $\alpha,\beta\in\mathbb{N}^n$, there exists a constant $C_{\alpha,\beta}>0$ satisfying
$$
\bigl|\partial_x^\alpha \partial_\xi^\beta a(x,\xi)\bigr|
\le C_{\alpha,\beta}(1+|\xi|)^{m-|\beta|}
\qquad \text{for all } (x,\xi)\in\mathbb{R}^n\times\mathbb{R}^n.
$$
An element $a \in S^m(\R^n \times \R^n)$ is called a \emph{symbol} of order $m$ on $\R^n$.

\item 
Given a symbol $a\in S^m(\R^n)$, define $\opa: \mathcal{S}(\mathbb{R}^n)\to C^\infty(\R^n)$ by 
$$
Op(a)u(x) = \frac{1}{(2\pi)^n} \int_{\mathbb{R}^n} e^{i x \cdot \xi} \, a(x, \xi) \, \hat{u}(\xi) \,\dd\xi,
$$
where $\hat u$ is the Fourier transform of $u$.
The operator $\opa$ is the \emph{pseudodifferential operator} associated to the symbol $a$, and it is of order $m$. 

\item 
Let $\Omega\subset\mathbb{R}^n$ be open. A $\pdo$ of order $m$ on $\Omega$ is a mapping $\A:C_c^\infty(\Omega)\to C^\infty(\Omega)$ such that for all $\phi,\psi\in C_c^\infty(\Omega)$, there exists a symbol $a_{\phi,\psi}\in S^m(\mathbb{R}^n\times\mathbb{R}^n)$ satisfying
$$
\phi \A \psi = {\rm Op}(a_{\phi,\psi}) \qquad \text{on } \mathbb{R}^n.
$$
Write $\A\in\Psi^m(\Omega)$.
\end{itemize}
\end{definition}

Note that for any $\A \in \Psi^s(\Omega)$ and every cutoff function $\chi \in C_c^\infty(\Omega)$, there exists a symbol $a_\chi(x,\xi) \in S^m(\mathbb{R}^n \times \mathbb{R}^n)$ such that $\chi A \chi = Op(a_\chi) + R$, where $R$ is a smoothing operator (with $C^\infty$-kernel). The total symbol $a_\chi$ is only uniquely determined modulo $S^{-\infty}:=\bigcap_{m \in \R} S^m$, but its leading order is well defined. In fact, by \cite[I.7.2.2]{AG} we have: 

\begin{defthm}\label{principal symbol} 
Let $\Omega \subset \R^n$ and $\A \in \Psi^m(\Omega)$ with symbol $a(x,\xi)\in S^m(\Omega\times\mathbb R^n)$ be as above. The \emph{principal symbol} of $\A$, denoted as $\ppl$, is the unique degree-$m$-homogeneous function in $\xi$ (for $|\xi| \ge 1$) such that
$$
a(x,\xi) - \sigma_{\mathrm{ppl}}(\A)(x,\xi) \in S^{m-1}(\Omega \times \mathbb{R}^n).
$$

\noindent
Let $\chi:\Omega\to\Omega'$ be a $C^\infty$-diffeomorphism between open subsets of $\mathbb R^n$. Then the pushforward of $\A$ under $\chi$ satisfies $\chi_*\A\in\Psi^m(\Omega')$, and 
\begin{align*}
\sigma_{\mathrm{ppl}}(\chi_*\A)\bigl(\chi(x),\eta\bigr) = \sigma_{\mathrm{ppl}}(\A)\bigl(x,{}^t D\chi(x)\,\eta\bigr) \quad \text{for all } (\chi(x),\eta)\in T^*\Omega'\setminus\{0\}.
\end{align*}

\end{defthm}

\subsection{$\pdo$s on manifolds}
Let $M$ be an $n$-dimensional $C^\infty$-manifold. By virtue of Definition/Theorem~\ref{principal symbol}, if $\A$ is a $\pdo$ on $M$ of order $m \in \R$ --- \emph{i.e.}, for any atlas $\{(\Omega_i, \varphi_i): i \in \mathcal{I}\}$ for $M$, the restriction $\A \circ \varphi_i^{-1}\big|_{\varphi_i(\Omega_i)}$ is in $\Psi^m(\varphi_i(\Omega_i))$ --- then the principal symbol $\sigma_{\mathrm{ppl}}(\A)$ is globally well-defined on $M$; in particular, the definition is independent of the choice of the atlas.

The definition of principal symbol extends naturally to $\pdo$s on vector bundles. We only outline the key ideas and properties here; see H\"{o}rmander \cite[Definition~18.1.32]{Hormander} for details.

Let $\pi_E: E\to M$ and $\pi_F:F\to M$ be vector bundles over a differentiable manifold $M$ of ranks $J$ and $I$, respectively. An order-$s$ $\pdo$  from $E$ to $F$, denoted $\A \in \Psi^s(M; E, F)$, is well-defined by localisation: $\A$ is a bundle map from $\G(E)$ to $\G(F^\C)$ such that for each local chart $\Omega \subset M$ trivialising both $E$ and $F$, $\A|_\Omega$ is an $\R^I$-valued $\pdo$ defined on $\pi_E^{-1}(\Omega)$ that leaves the $\Omega$-variables invariant. Here $F^\C := F \otimes_\R \C$ is the complexified bundle.

For a $\pdo$ $\A \in \Psi^s(M; E, F)$, the principal symbol is a smooth bundle map:
$$
\sigma_{\mathrm{ppl}}(\A): T^*M  \longrightarrow \operatorname{Hom}\left(E; F^\mathbb{C}\right).
$$
Thus, for each $(x,\xi)\in T^*M $,  $\sigma_{\mathrm{ppl}}(\A)(x,\xi)$ is a linear map from $E_x$ to $F_x\otimes\mathbb C$. Moreover, $\xi \mapsto \sigma_{\mathrm{ppl}}(\A)(x,\xi)$ is degree-$s$-homogeneous.

Also, recall Definition~\ref{def: cone} for the \emph{operator cone} of $\A \in \Psi^s(M; E, F)$. Heuristically, $\Lambda_\A$ contains the $E$-sections that are ``invisible'' by the principal symbol $\ppl$ in some directions.

\subsection{Semi-Riemannian manifolds}\label{subsec: semi-riem geom}
We refer the reader to \cite{ONeill} for a comprehensive treatment and to the following (see \cite[\S 2]{CL} also) for the elements needed for our subsequent developments.

A \emph{semi-Riemannian manifold} is a smooth manifold $M$ endowed with a smooth, symmetric, nondegenerate covariant $2$-tensor field $g \in \Gamma(T^*M \otimes T^*M)$, the metric tensor. 
The {\em index} of the semi-Riemannian metric $g$ on $T_xM$ is given by
\begin{equation*}
\text{Ind}(g;T_xM) := \max \{\dim V\,:\, V \subset T_xM \text{ is a vector subspace and $g|_V$ is negative definite}\}.
\end{equation*}
When $M$ is connected, $\text{Ind}(g;T_xM)$ is the same for all $x\in M$, so we write $\text{Ind}(g)$ instead. For a suitably small chart $\Omega \subset M$, we have an orthonormal basis  in which $g$ is diagonalised:
\begin{equation*}
g=\{g_{ij}\}=\delta_{ij}\epsilon_j|g_{ij}|\qquad \text{ for each } i,j\in\{1, 2,\ldots, n\}.
\end{equation*}
Here ${\epsilon}=(\epsilon_1,\ldots,\epsilon_n) \in \{-1,1\}^n$ is the {\em signature} of the metric $g$. As $g$ is non-degenerate, it has only nonzero entries on the diagonal, so $\text{Ind}(g)$ equals the number
of ``$-1$'' in signature $\epsilon$.

Let $(M^n,g)$ be a semi-Riemannian manifold. We define the volume form as $\dvg := \sqrt{|\det\,g|\,} \dd x$. Also note that $g(v,v)$ could be positive, negative, or zero, so one should define $|g(v,v)|:=|v|_g^2$ for the length of a vector field under  $g$. More generally, for vector bundles $(E,g^E)$ and $(F, g^F)$ as above with $g^E$, $g^F$ being semi-Riemannian bundle metrics, we may define the norm of $E$-sections ($F$-sections, resp.) with respect to both $g$ and $g^E$ ($g$ and $g^F$, resp.). To unburden the notation, we write $\|\bullet\|_{g^E}\equiv \|\bullet\|_{g^{E^\C}}$ and $\|\bullet\|_{g^F}\equiv \|\bullet\|_{g^{F^\C}}$.

The Sobolev space $H^s(M;E)$ of the $H^s$-regular $E$-sections for $s \in \R$ can be defined via a partition of unity argument (similarly for $H^s(M;F)$) --- On each trivialising chart $\Omega \subset M$, let $\varphi$ be the corresponding element in the partition of unity supported on $\Omega$. Then we can define $$\|\varphi v\|_{H^s(\Omega,g;\R^J,g^E)} \equiv \left\{\int_{\R^n} {\left|g^E\right|\left( \widehat{\varphi v}(\xi),\widehat{\varphi v}(\xi) \right)}{\left(1+|g(\xi,\xi)|^2\right)^s} \,\dd\xi\right\}^{\frac{1}{2}}.$$
Here, via the coordinate map on $\Omega$ we may view $\varphi v$ as a function on $\R^n$, identified with its extension-by-zero. If $\{\Omega_i\}_{i \in \mathcal{I}}$ is an atlas such that $E$ is trivialised over each $\Omega_i$ and $\{\varphi_i\}_{i \in \mathcal{I}}$ is a partition of unity subordinate to that, then set (well defined since  $\{\varphi_i\}_{i \in \mathcal{I}}$ is locally finite):
\begin{align*}
    \|v\|_{H^s(M,g; E, g^E)} := \sum_{i \in \mathcal{I}}\|\varphi_i v\|_{H^s(\Omega_i,g;\R^J,g^E)}.
\end{align*} 
Throughout, $|g^E|$ is the positive definite symmetric form that equals $g^E$ ($-g^E$, resp.) where $g^E$ is positive (negative, resp.) definite.

\subsection{Quadratic forms on a vector bundle} We adhere to the  convention in \cite[Definition~3.1]{CL}: By a quadratic form on $E$ we mean a \emph{bundle} map 
$$
Q:\Gamma(E)\longrightarrow C^\infty(M; \mathbb{C})
$$
such that at each $x \in M$, $Q_x$ is a quadratic polynomial on $E|_x$. More concretely, there exists  $
q\in\Gamma\bigl(\operatorname{Hom}(E\otimes E;\mathbb{C})\bigr)$ which is linear in the first argument and conjugate linear in the second, such that  for any $v\in\Gamma(E)$, $x\in M$, one has $ Q(v)\big|_x \equiv Q_x(v(x)) = q_x\bigl(v(x),v(x)\bigr)$.

In a trivialising chart $\Omega \subset M$ for $E$, write $\Phi: E^\C\big|_\Omega \to U \times \mathbb{C}^J$ for the trivialisation. For any local section $v \in \Gamma\left(E^\C\big|_\Omega\right)$, we have that
$$
\Phi(v(x)) = (x, \mathbf{v}(x)) \qquad \text{where } \mathbf{v}(x) = {}^t\left[v^1(x), \dots, v^J(x)\right] \in \mathbb{C}^J.
$$
Thus, the quadratic form $Q$ admits the local representation
\begin{align}\label{variable coefficients}
Q\big(v(x)\big) = \sum_{j,k=1}^J Q_{jk}(x)\, v^j(x)\, \overline{v^k(x)},
\qquad Q_{jk}\in C^\infty(\Omega).
\end{align}

As in \cite{Tartar}, we naturally extend $Q$ to the complexified bundle $E_{\mathbb{C}}$ by setting $Q^\C:=Q + iQ$. Similarly, we complexify the operator cone by $\Lambda_\A^{\mathbb C}:=\Lambda_\A+i\Lambda_\A$.

The following pointwise estimate for quadratic forms is crucial to the proof of Theorem~\ref{thm:psdo-cc}. Its proof is a direct adaptation of the arguments in \cite{Tartar}. We present it in the Appendix.

\begin{lemma}\label{lem:quadratic-inequality}
Let $(M^n,g)$ be a $C^\infty$-semi-Riemannian manifold, let $E$ and $F$ be real $C^\infty$-vector bundles over $M$ with semi-Riemannian metrics $g^E$ and $g^F$, let $\A\in \Psi^s(M;E,F)$ with $s>0$, and let $Q$ be a quadratic form on $E$. Assume  $\operatorname{Re} Q_x(v)\ge 0 $ for all $x\in M$ and $ v\in \Lambda_\A$. Then, for any compact set $\mathcal{K} \Subset T^*M\setminus\{0\}$ and any $\delta>0$, there exists a constant $C_{\delta,\mathcal{K}}>0$ such that for all $(x,\eta)\in \mathcal{K}$ and  $v \in E^\C$, it holds that
\begin{align}\label{ineq:pointwise}
\operatorname{Re} Q_x^{\mathbb C}(v(x)) \ge -\delta\,|v(x)|_{g^E}^2 -C_{\delta,\mathcal{K}}\,\bigl|\sigma_{\mathrm{ppl}}(\A)(x,\eta)v(x)\bigr|_{g^F}^2.
\end{align}
\end{lemma}

\section{Compensated compactness on Euclidean domains}\label{sec: Eucl}

This section is devoted to the proof of Theorem~\ref{thm:psdo-cc} when $\bar{u}=0$ and $M$, $E$, $F$ are all Euclidean spaces. More precisely, we prove the following.

\begin{proposition}\label{cor:euclidean-cc}
Let $\A \in \Psi^s(\mathbb R^n; \R^n\times \mathbb{R}^J, \R^n\times \mathbb{R}^I)$ be a pseudodifferential operator of order $s$, let $Q$ be a quadratic form on $\mathbb R^n \times \mathbb{R}^J$, and let  $\{v_\varepsilon\}$ be a sequence such that
\begin{enumerate}[label=(E\arabic*)]
\item $v_\varepsilon \rightharpoonup 0$ weakly in $L^2_{\mathrm{loc}}(\mathbb R^n; \R^n\times \mathbb{R}^J)$;
\item $\{\A v_\varepsilon\}$ is precompact in $H^{-s}_{\mathrm{loc}}(\mathbb R^n; \R^n\times \mathbb{R}^I)$;
\item $Q$ vanishes on the operator cone $\Lambda_\A$.
\end{enumerate}
Then $\lim_{\varepsilon \to 0} \int_{\mathbb R^n} Q\big(v_\varepsilon(x)\big) \psi(x) \,\dd x = 0$ for any test function $\psi \in C_c^\infty(\mathbb R^n)$.

\end{proposition}
\begin{proof}
\noindent
Our arguments are divided into six steps below.

\smallskip
\noindent
\textbf{Step~0. Reduction to the case of compactly supported sequences.} First, we observe that it suffices to assume that $\bigcup_{\e>0} {\rm supp}(v_\e) \subset \mathcal{K}_0$, where $\mathcal{K}_0$ is a compact subset of $\R^n$. Indeed, fix any $\psi \in C_c^\infty(\mathbb{R}^n)$ and assume that $\{v_\e\}$ satisfies conditions~$(E1)$--$(E3)$. Let us take $\tilde{v}_\varepsilon := \chi v_\varepsilon$, where $\chi \in C_c^\infty(\mathbb{R}^n)$ satisfies $\chi \equiv 1$ on $\operatorname{supp}(\psi)$. Then $(E1)$--$(E3)$\footnote{
The detailed justification can be found in Step~5 of the same proof. We do not repeat it here.} are also satisfied by $\{\tilde{v}_\varepsilon\}$, and $\int_{\mathbb{R}^n} Q\big(\tilde{v}_\varepsilon(x)\big) \psi(x) \,\mathrm{d}x = \int_{\mathbb{R}^n} Q\big(v_\varepsilon(x)\big) \psi(x) \,\mathrm{d}x$. We may thus prove the assertion for $\{\tilde{v}_\e\}$ in place of $\{v_\e\}$. In the sequel, without loss of generality, we take $v_\e \equiv \tilde{v}_\e$.

In Steps~1--4 below we prove $\lim_{\varepsilon \to 0} \int_{\mathbb R^n} Q\big(v_\varepsilon(x)\big) \,\dd x = 0$, and in Step~5 we shall indicate how to incorporate $\psi$ into this integral.


\smallskip
\noindent
\textbf{Step 1. Freezing of coefficients.} Fix an arbitrary constant $\gamma > 0$. 

Let $\mathcal{K}_0$ be the compact set as in Step~0. Cover $\mathcal{K}_0$  by $N_\gamma$ balls $\{B(x_\nu, r_\nu)\}_{\nu=1}^{N_\gamma}$ such that 
\begin{align}\label{oscillation}
&\max_{1\le j,k \le J} \sup_{x \in \bar{B}(x_\nu, r_\nu)} |Q_{jk}(x) - Q_{jk}(x_\nu)| < \gamma, \nonumber\\
&\sup_{x \in \bar{B}(x_\nu, r_\nu),\,|\xi|=1} |\sigma_{\mathrm{ppl}}(\A)(x, \xi) - \sigma_{\mathrm{ppl}}(\A)(x_\nu, \xi)| < \gamma.
\end{align}
Also, let $\{\phi_\nu\}_{\nu=1}^{N_\gamma}$ be a smooth partition of unity subordinate to $\{B(x_\nu, r_\nu)\}$.

Our goal is to show that the following integral converges to zero as $\e \to 0$:
$$
\int_{\mathbb R^n} Q\big(v_\varepsilon(x)\big) \,\dd x = \sum_{\nu=1}^{N_\gamma} \int_{\mathbb R^n} Q\big(v_\varepsilon(x)\big) \phi_\nu(x)\, \dd x.
$$
To this end, for each $\nu \in \{1,\ldots, N_\gamma\}$ we set $$w_{\varepsilon, \nu}(x) := \sqrt{\phi_\nu(x)} v_\varepsilon(x),$$ which is supported in $B(x_\nu, r_\nu)$. Define the \emph{constant-coefficient quadratic form} $Q_{x_\nu}$ by 
$$
Q_{x_\nu}\big(v(x)\big) := \sum_{j,k=1}^J Q_{jk}(x_\nu) v^j(x) \overline{v^k(x)} \qquad \text{for each } x \in \R^n,\, v: \R^n \longrightarrow \R^n \times \R^J.
$$
Then we have
\begin{align}\label{eq:total_freeze_error}
&\sum_{\nu=1}^{N_\gamma} \left| \int_{\mathbb R^n} Q(v_\varepsilon) \phi_\nu \,\dd x - \int_{\mathbb R^n} Q_{x_\nu}(w_{\varepsilon, \nu})\, \dd x \right|\nonumber\\
&\qquad \le \sum_{\nu=1}^{N_\gamma} \left(\sup_{x \in B_\nu,\, j,k \in \{1,\ldots, N_\gamma\}} |{Q}_{jk}(x) - {Q}_{jk}(x_\nu)|\right)\left( \int_{\mathbb R^n} \phi_\nu |v_\varepsilon|^2 \,\dd x \right)\nonumber \\
&\qquad\le \gamma \int_{\mathbb R^n} \left(\sum_{\nu=1}^{N_\gamma} \phi_\nu(x)\right) |v_\varepsilon(x)|^2 \,\dd x = \gamma \|v_\varepsilon\|_{L^2}^2,
\end{align}
so let us first argue that $\int_{\mathbb R^n} Q_{x_\nu}(w_{\varepsilon, \nu}(x))\, \dd x  \to 0$ as $\e \to 0$ for each $\nu \in \{1,\ldots, N_\gamma\}$. 

For notational convenience, from now on we fix $\nu$ and write $w_\varepsilon := w_{\varepsilon, \nu}$. We thus need to bound $\int_{\mathbb{R}^n} Q_{x_\nu}(w_\varepsilon(x)) \,\dd x$. It holds by the Plancherel theorem that
\begin{equation}\label{eq: plancherel}
    \int_{\mathbb{R}^n} Q_{x_\nu}\big(w_\varepsilon(x)\big) \dd x = \int_{\mathbb{R}^n} Q_{x_\nu}\big(\widehat{w}_\varepsilon(\xi)\big) \dd \xi.
\end{equation}
In the next two steps, we shall estimate the low- and high-frequency parts of the right-hand side.

Remark that Eq.~\eqref{eq: plancherel} is crucially based on the fact that $Q_{x_\nu}$ has constant coefficients.

\smallskip
\noindent
\textbf{Step 2. Low-frequency estimates.} In this step, we estimate
$\int_{|\xi| \le N} Q_{x_\nu}\big(\widehat{w}_\varepsilon(\xi)\big)\,\dd \xi$, where $N >0$ is a large constant to be specified later.

From now on, designate 
\begin{equation}\label{L}
    L_\nu:=\sup_\e\|w_\e\|_{L^2(\R^n)}  \equiv \sup_\e\|w_{\e,\nu}\|_{L^2(\R^n)},
\end{equation} 
and note that
\begin{equation}\label{L'}
    \sum_{\nu=1}^{N_\gamma} L_\nu^2 \leq \sup_\e \|v_\e\|^2_{L^2(\R^n)} =: L^2.
\end{equation}

Since ${\rm supp}(w_\e)\subset \mathcal{K}_0$, the Fourier transform $\widehat{w}_\varepsilon(\xi)$ is continuous and uniformly bounded: $$|\widehat{w}_\varepsilon(\xi)| \le \|w_\varepsilon\|_{L^1} \le C(\mathcal{K}_0) \|w_\varepsilon\|_{L^2} \leq C(\mathcal{K}_0) L_\nu.$$
Moreover, the weak convergence $w_\varepsilon \rightharpoonup 0$ in $L^2$ implies pointwise convergence:
$$
\widehat{w}_\varepsilon(\xi) = \int_{\mathcal{K}_0} w_\varepsilon(x) e^{-2\pi i x \cdot \xi}\, \dd x \longrightarrow 0 \quad \text{as } \varepsilon \to 0 \quad \text{for every fixed } \xi.
$$
Thus, by  the Dominated Convergence Theorem,
\begin{equation}\label{low-frequency estimate}
\lim_{\varepsilon \to 0} \int_{|\xi| \le N} Q_{x_\nu}\big(\widehat{w}_\varepsilon(\xi)\big) \dd \xi \le C \lim_{\varepsilon \to 0} \int_{|\xi| \le N} |\widehat{w}_\varepsilon(\xi)|^2 \dd \xi = 0.
\end{equation}

\smallskip
\noindent
\textbf{Step 3. High-frequency estimates.}
Introduce the ``compactification'' on the frequency space:
\begin{equation*}
\eta(\xi) := \frac{\xi}{(1+|\xi|^2)^{1/2}} \qquad \text{for } \xi \in \R^n.
\end{equation*}
Hence, 
\begin{equation}\label{K, new}
 \mathcal{K} :=\bigl\{(x,\eta(\xi)):\ x\in \mathcal{K}_0,\, \xi \in \R^n\bigr\}
\end{equation}
is a compact subset of $T^*\mathbb R^n\setminus\{0\}$, thus satisfying the hypotheses of Lemma~\ref{lem:quadratic-inequality}. 
Taking $\eta=\eta(\xi)$ and $v = \widehat{w}_\varepsilon(\xi)$ in Lemma~\ref{lem:quadratic-inequality} and integrating over $\{|\xi|\ge N\}$, we deduce that for every $\delta>0$, there exists $C_{\delta,\mathcal{K}}>0$ such that
\begin{align} \label{eq:high freq-Qx}
\int_{|\xi|\ge N} \operatorname{Re} Q_x^{\mathbb C}(\widehat w_\varepsilon(\xi))\,\dd\xi \ge -\delta \int_{|\xi|\ge N} |\widehat w_\varepsilon(\xi)|^2\,\dd\xi - C_{\delta,\mathcal K} \int_{|\xi|\ge N} \frac{|\sigma_{\mathrm{ppl}}(\A)(x,\xi)\widehat w_\varepsilon(\xi)|^2} {(1+|\xi|^2)^s}\,\dd\xi .
\end{align}
Here we use the fact that
\begin{align*}
\bigl| \sigma_{\mathrm{ppl}}(\A)(x,\eta(\xi))v \bigr|^2 = (1+|\xi|^2)^{-s} \bigl| \sigma_{\mathrm{ppl}}(\A)(x,\xi)v \bigr|^2,
\end{align*}
thanks to the definition of $\eta$ and the $s$-homogeneity of $\ppl$ in the frequency variable.


By \eqref{oscillation}, for each $x \in B(x_\nu, r_\nu)$ we have:
\begin{align*}
\left|Q_x^{\mathbb C}(w)-Q_{x_\nu}^{\mathbb C}(w)\right|= \left| \sum_{j,k=1}^J \bigl(Q_{jk}(x)-Q_{jk}(x_\nu)\bigr) w^j\overline{w^k} \right| \le \gamma \sum_{j,k=1}^J |w^j|\,|w^k| \;\le\; C_J\,\gamma\,|w|^2,
\end{align*}
where $C_J>0$ depends only on the dimension $J$. Thus, $
\operatorname{Re} Q_{x_\nu}^{\mathbb C}(w) \ge \operatorname{Re} Q_x^{\mathbb C}(w) - C_J \gamma |w|^2$, which together with \eqref{eq:high freq-Qx} implies that for each $x \in B(x_\nu, r_\nu)$,
\begin{small}
\begin{align}\label{eq:high freq}
\int_{|\xi|\ge N} \operatorname{Re} Q_{x_\nu}^{\mathbb C}(\widehat w_\varepsilon(\xi))\,\dd\xi
\ge -(\delta+C_J \gamma)\int_{|\xi|\ge N} |\widehat w_\varepsilon(\xi)|^2\,\dd\xi - C_{\delta,\mathcal K} \int_{|\xi|\ge N} \frac{|\sigma_{\mathrm{ppl}}(A)(x,\xi)\widehat w_\varepsilon(\xi)|^2}
{(1+|\xi|^2)^s}\,\dd\xi .
\end{align}
\end{small}

To proceed, we introduce the associated constant-coefficient operator:
\begin{equation}\label{A-nu}
    \A_{\nu} := \operatorname{Op}\bigl(\sigma_{\mathrm{ppl}}(\A)(x_\nu,\cdot)\bigr),
\end{equation}
which is a Fourier multiplier operator. Since $|a|^2\le 2|b|^2+2|a-b|^2$, we have that
\begin{small}
\begin{equation} \label{eq:triangle_split}
\left\|\frac{\sigma_{\mathrm{ppl}}(\A)(x,\xi)\widehat w_\varepsilon}{(1+|\xi|^2)^{s/2}}\right\|^2_{L^2(|\xi|\ge N)} \le 2\underbrace{\left\|\frac{\sigma_{\mathrm{ppl}}(\A)(x_\nu,\xi)\widehat w_\varepsilon}{(1+|\xi|^2)^{s/2}}\right\|^2_{L^2(|\xi|\ge N)}}_{=: [{\rm I}_\e]} + 2\underbrace{\left\|\frac{(\sigma_{\mathrm{ppl}}(\A)(x,\xi) - \sigma_{\mathrm{ppl}}(\A)(x_\nu,\xi))\widehat w_\varepsilon}{(1+|\xi|^2)^{s/2}}\right\|^2_{L^2(|\xi|\ge N)}}_{=: [{\rm II}_\e]}.
\end{equation}
\end{small}

Recall $L_\nu$ from \eqref{L}. We bound the right-hand side of \eqref{eq:triangle_split} by the two \emph{claims} below.

\begin{quote}
\textbf{Claim 1:} $\limsup_{\varepsilon \to 0} [{\rm I}_\e] \le [\gamma + N^{-1}C(\A,s)]^2 L_\nu^2$.
\begin{proof}[Proof of Claim~1]
In $[{\rm I}_\e]$, the principal symbol is ``frozen'' at the fixed point $x_\nu$. In view of the constant-coefficient $\pdo$ $\A_\nu$ introduced in \eqref{A-nu}, we have
\begin{equation*}
[{\rm I}_\e] = \int_{\mathbb{R}^n} \frac{|\sigma_{\mathrm{ppl}}(\A)(x_\nu,\xi)\widehat{w}_\varepsilon(\xi)|^2}{(1+|\xi|^2)^s}\1_{\{\|\xi| \geq N\}} \,\mathrm{d}\xi = \left\|\A_\nu w_\varepsilon\1_{\{\|\xi| \geq N\}}\right\|_{H^{-s}}^2.
\end{equation*} 
We estimate by triangle inequality:
\begin{align*}
\|A_{\nu} w_\varepsilon\1_{\{\|\xi| \geq N\}}\|_{H^{-s}} 
&= \bigl\| \A_{\nu} (\sqrt{\phi_\nu} v_\varepsilon) \1_{\{\|\xi| \geq N\}}\bigr\|_{H^{-s}} \\
&\le \underbrace{\bigl\|\sqrt{\phi_\nu} \A v_\varepsilon\bigr\|_{H^{-s}}}_{{\rm I}_{\e,1}} 
+ \underbrace{\bigl\|\left[\A, \sqrt{\phi_\nu}\right] v_\varepsilon\bigr\|_{H^{-s}}}_{{\rm I}_{\e,2}} 
+ \underbrace{\bigl\|(\A - \A_{\nu})w_\varepsilon\1_{\{\|\xi| \geq N\}}\bigr\|_{H^{-s}}}_{{\rm I}_{\e,3}}.
\end{align*}

\noindent
\underline{$\lim_{\e \to 0}{\rm I}_{\e,1} = 0$}: as $v_\e \weak 0$ in $L^2$ and $\{\A v_\e\}$ is precompact in $H^{-s}$, we have $\A v_\e \to 0$ in $H^{-s}$. The multiplication by $\sqrt{\phi_\nu}$ is bounded on $H^{-s}$.

\noindent
\underline{$\lim_{\varepsilon \to 0} {\rm I}_{\e,2}  = 0$}:  the commutator $\left[\A, \sqrt{\phi_\nu}\right] \in \Psi^{s-1}$, thus mapping $L^2$ continuously into $H^{-s+1}$, hence compactly into $H^{-s}$ by the Rellich lemma. We thus conclude by the assumption that $v_\varepsilon \weak 0$ in $L^2$.

\noindent
\underline{$\limsup_{\varepsilon \to 0} {\rm I}_{\e,3} \leq  [\gamma + N^{-1}C(\A,s)]^2 L_\nu^2 $}:  Since $w_\varepsilon$ is supported in $B(x_\nu, r_\nu)$, by \eqref{oscillation} one has
\begin{equation}\label{continuity-control}
\sup_{x \in \bar{B}(x_\nu, r_\nu), \, \xi \neq 0} \frac{|\sigma_{\mathrm{ppl}}(\A)(x,\xi)-\sigma_{\mathrm{ppl}}(\A)(x_\nu,\xi)|}{|\xi|^s} < \gamma.
\end{equation}
Thus, $\|(\A - \A_{\nu}) w_\varepsilon \1_{\{\|\xi| \geq N\}}\|_{H^{-s}} \leq \gamma L_\nu + \|Rw_\e\1_{\{\|\xi| \geq N\}}\|_{H^{-s}}$, where $R \in \Psi^{s-1}$. We conclude by observing that 
$$\|Rw_\e\1_{\{\|\xi| \geq N\}}\|_{H^{-s}} \leq \frac{1}{N} \|Rw_\e\1_{\{\|\xi| \geq N\}}\|_{H^{-s+1}} \leq \frac{C L_\nu}{N},$$ 
where $C= C(\A, s)$ by the $L^2 \to H^{-s+1}$-boundedness of $R$.  \end{proof}
\end{quote}

\begin{quote}
\textbf{Claim 2:} $\limsup_{\varepsilon \to 0} [{\rm II}_\e] \le  [\gamma + N^{-1}C(\A,s)]^2 L_\nu^2 $ for $N$ sufficiently large.
\begin{proof}[Proof of Claim~2] 
The proof is analogous to that of ${\rm I}_{\e,3}$ in Claim~1 above, by using \eqref{continuity-control} and Plancherel.  \end{proof}
\end{quote}

Putting together \emph{Claims}~1 $\&$ 2 and \eqref{eq:triangle_split}, we obtain for $N$ large enough that
\begin{align*}
\limsup_{\e \to 0} \left\|\frac{\sigma_{\mathrm{ppl}}(\A)(x,\xi)\widehat w_\varepsilon}{(1+|\xi|^2)^{s/2}}\right\|_{L^2(|\xi|\ge N)}^2 \leq  2[\gamma + N^{-1}C(\A,s)]^2 L_\nu^2.
\end{align*}
Hence, by \eqref{eq:high freq},  
$$
\liminf_{\varepsilon \to 0} \int_{|\xi| \ge N} \operatorname{Re} Q_0^{\mathbb C}\big(\widehat{w}_\varepsilon(\xi)\big) \dd \xi 
\ge -\Bigl[(\delta+C_J \gamma) +  2[\gamma + N^{-1}C(\A,s)]^2  C_{\delta,\mathcal K} \Bigr]
L_\nu^2 .
$$
Summing over $\nu \in\{1, \dots, N_\gamma\}$ and recalling \eqref{L'}, we deduce that
\begin{align}
\liminf_{\varepsilon \to 0}\sum_{\nu=1}^{N_\gamma}  \int_{|\xi|\ge N} \operatorname{Re} Q_{x_\nu}^{\mathbb C}\big(\widehat w_\varepsilon(\xi)\big)\,\dd\xi
\ge -\left[(\delta+ C_J \gamma) + 2[\gamma + N^{-1}C(\A,s)]^2  C_{\delta,\mathcal K}\right] L^2. \label{high-frequency estimate}
\end{align}

\smallskip
\noindent
\textbf{Step 4. Conclusion for compactly supported sequences.} Combining \eqref{eq:total_freeze_error}, \eqref{low-frequency estimate}, and \eqref{high-frequency estimate}, we obtain that for any $\delta, \gamma>0$, there exist $C_0=C(\mathcal{K}_0)$ and $C_{\delta, \mathcal{K}}>0$ such that
\begin{align*}
    \liminf_{\varepsilon \to 0}  \int_{\mathbb R^n} \operatorname{Re} Q^{\mathbb C}\big(v_\varepsilon(x)\big) \,\dd x &\geq - C_0\gamma L^2 - \left[(\delta+ C_J \gamma) + 2[\gamma + N^{-1}C(\A,s)]^2  C_{\delta,\mathcal K}\right] L^2\\
    &\geq -\bigg\{ \delta + (C_0+C_J)\gamma + 4 C_{\delta, \mathcal{K}}\gamma^2 + \frac{4 C(\A,s)^2C_{\delta, \mathcal{K}}}{N^2} \bigg\}L^2.
\end{align*}
Recall from \eqref{L'} that $L = \sup_\e \|v_\e\|_{L^2}<\infty$. Also, by the definition of $\mathcal{K}$ in \eqref{K, new} and Lemma~\ref{lem:quadratic-inequality},  $C_{\delta, \mathcal{K}}$ depends only on $\delta$ and $\mathcal{K}_0$ (the common support of $v_\e$).

Given an arbitrary $\kappa > 0$, first select $\delta > 0$ so small that $\delta L^2 < \kappa/4$.  With $\delta$ fixed, $C_{\delta, \mathcal{K}}$ is also  determined. Then choose $N=N(\A,s,\delta)$ such that $ {4 C(\A,s)^2C_{\delta, \mathcal{K}}}{N^{-2}}L^2 < \kappa/4$. Finally, choose $\gamma>0$ such that  $(C_0+C_J)\gamma < \kappa/4$ and $4 C_{\delta, \mathcal{K}}\gamma^2 < \kappa/4$. Thus, $
\liminf_{\varepsilon \to 0} \int_{\mathbb R^n} \operatorname{Re} Q^{\mathbb C}(v_\varepsilon) \,\mathrm{d}x \ge  -\kappa$. 
By the arbitrariness of $\kappa > 0$ and applying the same argument to $-Q$, we conclude that $\lim_{\varepsilon \to 0} \int_{\mathbb R^n} \operatorname{Re} Q^{\mathbb C}\big(v_\varepsilon(x)\big)\, \dd x = 0$. Here $v_\varepsilon$ is real-valued, thus $$\lim_{\varepsilon \to 0} \int_{\mathbb{R}^n} Q\big(v_\varepsilon(x)\big) \,\dd x = 0,$$ \emph{provided that} $\{v_\varepsilon\}$ are compactly supported on $\mathcal{K}_0 \Subset \R^n$.

\smallskip
\noindent
\textbf{Step 5. Proof for $\{v_\e\}\subset L^2_\loc$.}
Finally, we prove for the general case that $v_\e$ are only in $L^2_\loc$, not necessarily compactly supported.



Let $\psi \in C_c^\infty(\mathbb{R}^n)$ be a fixed test function as in the statement of Proposition~\ref{cor:euclidean-cc}. Take any $\theta \in C_c^\infty(\mathbb R^n)$ such that $\theta \equiv 1$ on $\operatorname{supp}(\psi)$. Hence, for any $x \in \R^n$,
\begin{equation*}
\psi(x) = \psi(x)\theta(x) = \psi_1(x)^2 - \psi_2(x)^2\quad \text{ where } \psi_1 := \frac{1}{2}(\psi + \theta), \quad  \psi_2 := \frac{1}{2}(\psi - \theta).
\end{equation*}
Note that $\psi_1, \psi_2 \in C_c^\infty(\mathbb R^n)$. As $Q$ is quadratic, we have
$$
\int_{\mathbb R^n} Q\big(v_\varepsilon(x)\big) \psi(x)\, \dd x = \int_{\mathbb R^n} Q\big(\psi_1 v_\varepsilon(x)\big)\,  \dd x  - \int_{\mathbb R^n} Q\big(\psi_2 v_\varepsilon(x)\big)\,  \dd x.
$$
Therefore, once we show that if $\{v_\e\}$ satisfies Conditions~$(E1)$--$(E3)$, then so does $\{\tilde{v}_\varepsilon\}$ with $\tilde{v}_\varepsilon = \psi v_\varepsilon$ for arbitrary $\psi \in C_c^\infty(\mathbb R^n)$, we can  conclude the proof  from Steps~0--4 above.

For this purpose, observe that $(E1)$ and $(E3)$ are immediate. To see that $(E2)$ is preserved under multiplication by $\psi \in C_c^\infty(\mathbb R^n)$, write 
$$
\A(\psi v_\varepsilon)=\psi \A v_\varepsilon+[\A,\psi]v_\varepsilon.
$$
As $\{v_\e\}$ satisfies $(E2)$, $\psi \A v_\varepsilon$ is precompact in $H^{-s}_{\mathrm{loc}}$. The commutator $[\A, \psi]$ is a $\pdo$ of order at most $s-1$. Thus, by \cite[Proposition~5.2]{AG}, it maps $L^2_{\mathrm{loc}}$ continuously into $H^{-s+1}_{\mathrm{loc}}$, hence compactly into $H^{-s}_{\mathrm{loc}}$ by the Rellich lemma.

This completes the proof.  \end{proof}

\section{Compensated Compactness on Manifolds}\label{sec: mfd}


In this section, we deduce Theorem~\ref{thm:psdo-cc} from Proposition~\ref{cor:euclidean-cc}. 

\begin{proof}
The proof is divided into three steps.

\smallskip
\noindent\textbf{Step~0: Reduction to $\bar u = 0$.}
Let $q\in \Gamma(\operatorname{Hom}(E\otimes E;\mathbb C))$ be the sesquilinear form associated with $Q$, namely $ Q(u)=q(u,u)$. Hence,
$$
Q(u_\varepsilon)= Q(u_\varepsilon-\bar u) + Q(\bar u) + 2\,\operatorname{Re}\, q(u_\varepsilon-\bar u,\bar u).
$$
The weak convergence $u_\varepsilon \rightharpoonup \bar u$ in $L^2_{\mathrm{loc}}$ implies that $\operatorname{Re}\, q(u_\varepsilon-\bar u,\bar u) \to 0$ in the sense of distributions. Thus,
$$
\lim_{\varepsilon \to 0} \left( \int_M Q(u_\varepsilon)\psi \,\dd V_g - \int_M Q(\bar u)\psi \,\dd V_g \right) = \lim_{\varepsilon \to 0} \int_M Q(u_\varepsilon - \bar u)\psi \,\dd V_g.
$$

From now on, we may assume without loss of generality that $\bar u\equiv 0$.

\smallskip
\noindent\textbf{Step~1: Reduction to a single coordinate chart.}  Let $\{U_k\}$ be an atlas for $M$ and $\{\varphi_k\}$ be a smooth partition of unity subordinate to it. 
Suppose that the theorem holds for each chart. As in Step~5 of the proof of Proposition~\ref{cor:euclidean-cc}, if $\{u_\e\}$ satisfies Condition~$(C1)$--$(C3)$, then so does $\{\sqrt{\varphi_k} u_\e\}$ for fixed $k$.  Hence, for any test function $\psi\in C_c^\infty(M)$, 
\begin{align*}
    \lim_{\varepsilon\to0}\int_M Q\big(u_\varepsilon(x)\big)\psi(x)\,\mathrm{d}V_g(x)&=\sum_{k\in\mathbb{N}}\lim_{\varepsilon\to0}\int_{U_k}\varphi_k(x)Q\big(u_\varepsilon(x)\big)\psi(x)\,\mathrm{d}V_g(x) \\
    &=\sum_{k\in\mathbb{N}}\lim_{\varepsilon\to0}\int_{U_k}Q\big(\sqrt{\varphi_k}(x)u_\varepsilon(x)\big)\psi(x)\,\mathrm{d}V_g(x) \\
    &=\sum_{k\in\mathbb{N}}\int_{U_k}Q\big(\sqrt{\varphi_k}(x)\bar u(x)\big)\psi(x)\,\mathrm{d}V_g(x)=0.
\end{align*}

Therefore, it suffices to prove the assertion on a fixed chart $(U_k \equiv U,g)$. 

\smallskip
\noindent\textbf{Step~2: Reduction to the Euclidean case.} Assume that the sequence $\{u_\varepsilon\}$ and the test function $\psi$ are compactly supported on a single coordinate chart $(U, \Phi)$, where $\Phi: U \xrightarrow{\sim} \Omega \subset \mathbb{R}^n$ is the coordinate map. In addition, assume that both bundles $E$ and $F$ are trivialised over $U$.

Set $v_\varepsilon := u_\varepsilon \circ \Phi^{-1}$ and $\Upsilon := (\psi \circ \Phi^{-1})\sqrt{|\det\,g|\circ \Phi^{-1}\,}$. Identifying $u_\e$, $\psi$ with their extension-by-zero outside their compact supports and tacitly composing with the trivialising maps $E|_U \cong U \times \R^J$ and $F|_U \cong U \times \R^I$, we may view $v_\e: \R^n \to \R^n \times \R^J$ and $\Upsilon \in C_c^\infty(\R^n;\R)$; meanwhile, $\Phi_*\A \in \Psi^s(\R^n; \R^n \times \R^J, \R^n \times \R^I)$.

Denote by $\euc$ the Euclidean metric of any dimension. We now verify Conditions~$(E1)$--$(E3)$ in Proposition~\ref{cor:euclidean-cc} for $\{v_\e\}$ and $\Phi_*\A$, providing that $(C1)$--$(C3)$ hold for $\{u_\e\}$ and $\A$.

\begin{itemize}
    \item 
     By $(C1)$ we have $u_\varepsilon \rightharpoonup 0$ in $L^2(U, g;E,g^E)$. Since the semi-Riemannian metric $g$ is non-degenerate, it implies that $u_\varepsilon \rightharpoonup 0$ in $L^2(U,\euc; E,g^E)$. It is crucial to note that
     \begin{align*}
         \|v_\e\|^2_{L^2(U,\euc; E,g^E)} = \int_U \left|g^E\right|_{\alpha\beta} v_\e^\alpha v_\e^\beta \,\dd x.
     \end{align*}
Here $\left|g^E\right|_{\alpha\beta} \equiv \e_\alpha g^E_{\alpha\beta}$ (without summation convention), and $\e_\alpha\in\{\pm 1\}$ are the signature components of $g^E$ as in \S\ref{subsec: semi-riem geom}. As $\left|g^E\right|$ is positive definite (thanks to the non-degeneracy of $g^E$), one deduces that $v_\e \weak 0$ in $L^2(U,\euc; E = \R^n \times \R^J,\euc)$, which is $(E1)$.

     \item 
Given the $\pdo$ $\A\in\Psi^s(M;E,F)$, its principal symbol transforms as in \cite[Theorem 18.1.17]{Hormander}:
$$
\sigma_{\mathrm{ppl}}(\Phi_*\A)\bigl(y,\zeta\bigr)= \sigma_{\mathrm{ppl}}(\A)\Bigl(\Phi^{-1}(y), {}^t\!D\Phi\big(\Phi^{-1}(y)\big)^{-1}\zeta\Bigr),\qquad (y,\zeta)\in T^*\mathbb{R}^n\setminus\{0\}.
$$
In particular, the kernel of the principal symbol remains invariant. More precisely, $v \in \Lambda_{\Phi_*\A}(y)$ if and only if $v \circ \Phi^{-1} \in \Lambda_\A(\Phi^{-1}(y))$. This gives us $(E3)$.

\item 
Finally, we show that $\Phi_*\A(v_\e)$ is precompact in $H^{-s}(U,\euc; F \cong \R^n \times \R^I, \euc)$. Indeed,
\begin{align*}
&\left\|\Phi_*\A(v_\e)\right\|_{H^{-s}(U,\euc; F \cong \R^n \times \R^I, \euc)} := \int_{\R^n} \frac{\left|\widehat{\Phi_*\A(u_\e\circ \Phi^{-1})}(\xi)\right|_\euc^2}{\left(1+|\xi|_\euc^2\right)^{s/2}}\,\dd\xi \\
& \qquad\qquad\lesssim \int_{\R^n} \frac{\left|g^F\right|\left(\widehat{\A(u_\e)}(\xi),\widehat{\A(u_\e)}(\xi)\right)}{\left(1+|\xi|_g^2\right)^{s/2}}\,\dd\xi =: \left\|\A u_\e\right\|_{H^{-s}(U,g; F \cong \R^n \times \R^I, g^F)},
\end{align*}
modulo some constant depending only on $s$, $\|g\|_{C^1}$, $\|\Phi\|_{W^{s+1,\infty}}$, and the smallest eigenvalue of $\left|g^F\right|$. The right-most term converges to zero due to $(C2)$, so does the left-most term, which yields $(E2)$.

\end{itemize}


With Conditions~$(E1)$--$(E3)$ verified, we may now apply Proposition~\ref{cor:euclidean-cc} to deduce that
$$
\lim_{\varepsilon\to 0} \int_{\R^n} Q\big(v_\varepsilon(x)\big)\Upsilon(x) \,\mathrm{d}x = \lim_{\varepsilon\to 0} \int_{\Phi(U)} Q\Big(u_\e\big(\Phi^{-1}(x)\big)\Big) \psi\big(\Phi^{-1}(x)\big) \sqrt{\left|\det\,g\big(\Phi^{-1}(x)\big)\right|} \,\dd x=0
$$
for every $\psi\in C_c^\infty(M)$. By a change of variables, this is equivalent to
\begin{align*}
  \lim_{\varepsilon\to 0} \int_{U} Q\big(u_\varepsilon(x)\big)\psi(x) \,\mathrm{d}V_g = 0,
\end{align*}
which completes the proof.
\end{proof}

\appendix
\section{Proof of Lemma~\ref{lem:quadratic-inequality}}
\begin{proof}
We argue by contradiction. Suppose that there were a compact set $\mathcal{K} \Subset T^*M\setminus\{0\}$, $\delta>0$ and sequences
$$
\{(x_j,\eta_j)\}\in \mathcal{K}, \qquad \{v_j(x_j)\}\in E_{x_j}^\C, \qquad C_j\to+\infty,
$$ 
such that
\begin{align}\label{eq:contradiction}
\operatorname{Re}Q_{x_j}^{\mathbb C}(v_j) < -\delta\,|v_j|_{g^E}^2 - C_j\,\bigl|\sigma_{\mathrm{ppl}}(\A)(x_j,\eta_j)v_j\bigr|_{g^F}^2 \qquad\text{for all }j \in \mathbb{N}.
\end{align}
Set $w_j := v_j / |v_j|_{g^E}$. Dividing \eqref{eq:contradiction} by $|v_j|_{g^E}^2$, we obtain that
\begin{align}\label{eq:normalized}
\operatorname{Re}Q_{x_j}^{\mathbb C}(w_j)
< -\delta - C_j\,\bigl|\sigma_{\mathrm{ppl}}(A)(x_j,\eta_j)w_j\bigr|_{g^F}^2.
\end{align}

Since $\mathcal{K}$ is compact, there exists a subsequence (not relabeled) such that
$$
(x_j,\eta_j)\longrightarrow (x_0,\eta_0)\in \mathcal{K} \qquad \text{as } C_j \to \infty.
$$
As the unit sphere in $E^\C_{x_j} \cong \mathbb{C}^J$ is compact, there exists a further subsequence of $w_j$ that converges to a unit vector $w_0 \in E_{x_0}^\mathbb{C}$ with $|w_0|_{g^E}=1$. On the other hand, the principal symbol $\sigma_{\mathrm{ppl}}(\A)(x,\eta)$ and the quadratic form $Q_x$ depend smoothly on $(x,\eta)$. Thus,
$$
\bigl|\sigma_{\mathrm{ppl}}(A)(x_j,\eta_j)w_j\bigr|_{g^F}
\longrightarrow
\bigl|\sigma_{\mathrm{ppl}}(A)(x_0,\eta_0)w_0\bigr|_{g^F}.
$$
As $C_j\to+\infty$, inequality \eqref{eq:normalized} forces
$$
\bigl|\sigma_{\mathrm{ppl}}(A)(x_0,\eta_0)w_0\bigr|_{g^F}=0,
$$
thus $w_0 \in \Lambda_\A(x_0)$. 

However, by the assumption in Lemma~\ref{lem:quadratic-inequality}, we have $\operatorname{Re}Q_{x_0}^{\mathbb C}(w_0)\ge 0$. However, passing to the limit in \eqref{eq:normalized} yields
$$
\operatorname{Re}Q_{x_0}^{\mathbb C}(w_0)\le -\delta<0.
$$
Contradiction. 
\end{proof}

\noindent
{\bf Acknowledgement}. 
SL is supported by NSFC Projects 12201399, 12331008, and 12411530065, Young Elite Scientists Sponsorship Program by CAST  2023QNRC001, the National Key Research $\&$ Development Programs 2023YFA1010900 and 2024YFA1014900, Shanghai Rising-Star
Program 24QA2703600, and the Shanghai Frontier Research Institute for Modern Analysis.

The research of XS is partially supported by the National Key Research $\&$ Development Programs 2023YFA1010900 and 2024YFA1014900. 

\medskip
\noindent
{\bf Competing Interests Statement}. We declare that there is no conflict of interests involved.

\end{document}